\documentclass[a4paper, 11pt]{article}
\usepackage[utf8]{inputenc}
\usepackage[T1]{fontenc}
\usepackage[english]{babel}
\usepackage{graphicx}
\usepackage{stmaryrd}
\usepackage{tikz}
\usepackage{tikz-cd} 
\usepackage[all]{xy}
\usepackage{comment}
\usepackage{bm}
\usepackage{comment}
\usepackage{amsmath,amsfonts,amssymb}
\usepackage[a4paper]{geometry}
\usepackage{amsthm}
\usepackage{float}
\usepackage{enumitem}
\usepackage{hyperref}
\usepackage{xcolor}
\hypersetup{
    colorlinks=true,
    linkcolor={blue},
    citecolor={blue},
    urlcolor={blue}}
\newtheorem{te}{Theorem}[section]
\newtheorem{prop}[te]{Proposition}

\newtheorem{co}[te]{Corollary}

\newtheorem{lemme}[te]{Lemma}

\theoremstyle{definition}

\newtheorem{de}[te]{Definition}
\newtheorem{ex}[te]{Example}

\theoremstyle{remark}
\newtheorem{rque}[te]{Remark}
\newlength{\plarg}
\usetikzlibrary{cd}
\renewcommand{\phi}{\varphi} 
\newcommand{\m} {^{-1}}

\newcommand {\ra} {\rightarrow}

\newcommand {\onto} {\twoheadrightarrow}

\newcommand {\xra} {\xrightarrow}

\newcommand{\actson}{\curvearrowright}

\newcommand{\semidirect}{\ltimes}

\newcommand{\ie} {i.e.\ }

\newcommand {\calc} {{\mathcal {C}}}

\newcommand {\cali} {{\mathcal {I}}}

\newcommand {\bbC} {{\mathbb {C}}}

\newcommand {\bbF} {{\mathbb {F}}}   
   
\newcommand {\bbH} {{\mathbb {H}}}

\newcommand {\bbN} {{\mathbb {N}}}

\newcommand {\bbQ} {{\mathbb {Q}}}   
\newcommand {\bbR} {{\mathbb {R}}}

\newcommand {\bbZ} {{\mathbb {Z}}}

\newcommand{\grp}[1]{\langle #1 \rangle}

\newcommand{\Comm} {{\mathrm{Comm}}}

\newcommand{\ad} {{\mathrm{ad}}}


\newcommand{\Cent}{\mathrm{Cent}}

\newcommand{\xonto}[2][]{%
  \xrightarrow[#1]{#2}\mathrel{\mkern-14mu}\rightarrow
}

\title{Finitely generated simple sharply 2-transitive groups}
\author{Simon André and Vincent Guirardel}
\date{\today}

\begin{document}
\begin{minipage}{\linewidth}

\maketitle

\vspace{0mm}

\begin{abstract}
We construct the first examples of infinite sharply 2-transitive groups which are finitely generated. Moreover, we construct such a group that has Kazhdan property (T), is simple, has exactly four conjugacy classes, and we show that this number is as small as possible.
\end{abstract}

\end{minipage}

\vspace{5mm}

\section{Introduction}

\thispagestyle{empty}

An action of a group $G$ on a set $X$ containing at least two elements is said to be \emph{sharply $2$-transitive} if, for any two couples $(x_1,x_2)$ and $(y_1,y_2)$ of distinct elements of $X$, there exists a unique element of $G$ mapping $(x_1,x_2)$ to $(y_1,y_2)$. A group $G$ is called \emph{sharply $2$-transitive} if there exists a set $X$ on which $G$ acts sharply $2$-transitively. 
For instance, if $K$ is any skew field, the natural action of the affine group $K^*\ltimes K$ on $K$ is sharply 2-transitive. This generalizes to a near-field $K$ where only one of the two distributive rules is required (see for instance \cite[\S 7.6]{Dixon_Mortimer}). 
We say that $G$ is \emph{nearly affine} if $G\simeq K^*\ltimes K$ for some near-field $K$ (this notion is also referred to as \emph{split} in the literature).

Zassenhaus proved that every finite sharply $2$-transitive group $G$ is 
nearly affine,
%isomorphic to the affine group $K^*\ltimes K$ of a finite near-field $K$, 
and classified all finite near-fields in \cite{Z1,Z2}.
In the same direction, Tits proved in \cite{Tits2,Tits3} that if $G$ is a locally compact connected group having a continuous sharply 2-transitive action on a locally compact set,
then $G$ is nearly affine. 
%$G\simeq K^*\ltimes K$ for some near-field $K$. 
In fact, $G\simeq K^*\ltimes K$ where $K$ is either
$\bbR$, $\bbC$, $\bbH$ or a near-field obtained by twisting the multiplication on $\bbH$. More recently,
Glasner and Gulko \cite{GG2} and Glauberman, Mann and Segev \cite{GMS} proved that if $G$ is any linear group having a sharply 2-transitive action of characteristic distinct from 2 (see Subsection \ref{characteristic} for a definition),
then again, $G$ is nearly affine. 
%See also \cite{GiMa} for sharply 2-transitive groups of homeomorphisms of the circle.
%$G\simeq K^*\ltimes K$ for some near-field $K$.

Until recently, it was an open problem whether every sharply 2-transitive group $G$ is 
nearly affine.
%isomorphic to $G\simeq K^*\ltimes K$ for some near-field $K$.
The first counterexamples were constructed in \cite{RST} and \cite{RT} by Rips, Segev and Tent. 
Then, infinite simple sharply 2-transitive groups were constructed in \cite{AT21}. 

But none of the examples constructed in \cite{RST,RT,AT21} is finitely generated,
so one may ask whether Zassenhaus classification extends assuming finite generation:
if $G$ is a finitely generated sharply 2-transitive group, does $G$ have to be nearly affine?

In fact, there is an even more basic question: 
does there exist an infinite, sharply 2-transitive group $G$ which is finitely generated?
%In fact, this question may sound strange as there is no known example of
%and the problem of the existence of an infinite sharply 2-transitive group which is finitely generated is open.
Indeed, no group of the form $G=K^*\semidirect K$ for a skew field $K$ is known to be finitely generated as the existence of an infinite skew field $K$ whose multiplicative group is finitely generated is a famous open problem
(see for instance \cite{AkMa,AKPR}).

%Thus, there was no example of a near-domain whose group of affine transformations is finitely generated.

%Moreover, none of the examples constructed in \cite{RST,RT,AT21} is finitely generated,
%so the problem whether sharply 2-transitive groups are nearly affine is open for 
%finitely generated groups.

In the present paper, we answer the two questions above:
we construct the first examples of infinite sharply 2-transitive groups which are finitely generated, and our examples are not nearly affine.
These are also the first examples of such infinite groups having Kazhdan property (T).
%Moreover, these groups do not come from a 
%Moreover, these are the first examples of such groups which are infinite and have Kazhdan property (T).

\begin{te}\label{thm_fg}
There exists an infinite, finitely generated, sharply 2-transitive group $G$. 
Moreover, one can construct such a group $G$ which has property (T), is 2-generated and simple (and in particular, not nearly affine).
\end{te}

\begin{rque}
The sharply 2-transitive groups constructed in this paper have characteristic 0, meaning that involutions have a fixed point and that the products of distinct involutions have infinite order
(see Subsection \ref{characteristic} for more details). % for a definition of the characteristic of a sharply 2-transitive group).
\end{rque}

%\begin{rque}
% It is well known (see \cite[Chapter~II \S 6]{Kerby}) that every sharply 2-transitive group is the group of affine transformations of a \emph{near-domain} (which is not a semi-direct product in general). Thus, Theorem \ref{thm_fg} can be reformulated as follows, and can be compared to the question of the existence of skew fields with finitely generated
% multiplicative group (or equivalently, with finitely generated affine group).
% 

In fact, it is well known (see \cite[Chapter~II \S 6]{Kerby}) that every sharply 2-transitive group is the group of affine transformations of a \emph{near-domain}: this is a structure $(K,+,\cdot)$ similar to a near-field but in which $(K,+)$ is not required to be a commutative group but only a loop
with a twisted associativity condition.
%. In this context, %and for which
In this context, the group $G$ of affine transformations $x\mapsto ax+b$ of $K$ acts sharply 2-transitively on $K$, but it is not a semi-direct product in general. Actually, $G$ is a semi-direct product if and only if $K$ is a near-field, in which case $G$ is nearly affine (see \cite[Theorem 7.1]{Kerby}).
Thus, examples of exotic sharply 2-transitive groups yield exotic near-domains, and Theorem \ref{thm_fg} can be reformulated in terms of near-domains as follows.

\begin{co}
There exists an infinite near-domain $K$ which is not a near-field and whose associated group of affine transformations 
is finitely generated. 
\end{co}

We also consider the problem of minimizing the number of conjugacy classes.
% We observe that in characteristic distinct from 2, any sharply 2-transitive group has at least four conjugacy classes (except for the single exception $G=S_3$, see Subsection \ref{optimal_number} for details), 
First, any sharply 2-transitive group has to have at least three conjugacy classes (except for $\bbZ/2\bbZ$), 
and at least four conjugacy classes
in characteristic distinct from $2$ (except for the affine group on $\bbF_3$), 
%exceptions \Vmodif $G=\sym{2}$ or $\sym{3}$), 
see Proposition \ref{optimal}. 
Actually, except for the affine groups on $\bbF_2$, $\bbF_3$ and $\bbF_4$, 
all \emph{finite} sharply 2-transitive groups have at least five conjugacy classes. 

% \textcolor{blue}{Je suis d'accord avec ta demande de nuance. Mais je ne comprends plus.
% J'allais dire:
% "On the other hand, for each prime power $q=p^r>0$, Cohn constructed an infinite skew field of characteristic $p$
% with exactly $q$ conjugacy classes (ou plus simplement le meme enonce avec $q=p$)."
% En fait je vois pas comment deduire ça du thm 6.3 ou du corollaire 1 (encore moins 2) de Cohn, en fait meme pour q=2.
% Si on part du corps de fractions rationnelles $F_p(X)$, le theoreme 6.3 n'interdit pas qu'on rajoute des elements algebriques quand on fait nos HNN,
% ce qui fait que meme si tous les elements ayant le meme polynome minimal sont conjugues, peut-etre qu'on a une infinité de polynomes minimaux qui apparaissent.
% C'est vraiment vrait que le thm implique qu'il y a un corps dont le groupe multiplicatif a 2 classes de conj ?
% }

On the other hand, Cameron \cite{Cameron_classical} mentions that Cohn's construction \cite[Theorem 6.3]{Coh71} can be used
to construct an infinite skew field $K$ of characterisic $p$ such that 
$K^*$ has exactly $p$ conjugacy classes.
For $p=2$, the corresponding affine group $K^*\ltimes K$ would have exactly three conjugacy classes.
% whose multiplicative group $K^{*}$ has exactly two conjugacy classes  
% (see Corollary 2 of Theorem 6.3 in \cite{Coh71}), and its corresponding affine group $K^*\ltimes K$ has three conjugacy classes.
% But no such example is known in other characteristics\Scomment{je nuancerais un peu, parce que la construction de Cohn produit un skew field $K$ de n'importe quelle carac $p<\infty$ tel que $K^*$ a $p$ classes de conjugaison, et donc $GA(K)$ en a $p+1$ (voir détails ligne 435 en gris).}, and 
%in all the previously known sharply 2-transitive groups of characteristic 0, the number of conjugacy classes is infinite. 
In characteristic 0, all the previously known sharply 2-transitive groups have infinitely many conjugacy classes.

Also note that in the realm of abstract groups, the first infinite, finitely generated groups with only finitely many conjugacy classes were constructed by Ivanov, and a construction can be found in \cite{Ol91} (see Theorem 41.2). This construction yields for each sufficiently large prime number $p$ 
a group of exponent $p$ with exactly $p$ conjugacy classes. Then, Osin constructed in \cite{Osi10} the first infinite finitely generated groups with exactly $n$ conjugacy classes for any natural number $n\geq 2$.

% We observe that in characteristic distinct from 2, any sharply 2-transitive group has at least four conjugacy classes (except for the single exception $G=S_3$, see Subsection \ref{optimal_number} for details), 
Our construction, relying on Osin's small cancellation techniques \cite{Osi10}, 
provides examples of groups which are sharply 2-transitive 
in characteristic 0, are finitely generated, and have 
the smallest possible number of conjugacy classes.
%exactly four conjugacy classes.
%and we prove that the number four can actually be reached.

%\begin{qu}We denote by $c(G)$ the number of conjugacy classes in a group $G$.\begin{enumerate}\item Is there a S2T group $G$ of characteristic 0 with $c(G)<\infty$, or even with $c(G)=4$?\item Is there an infinite simple S2T group $G$ with $c(G)<\infty$, or even with $c(G)=4$?\end{enumerate}\end{qu}We will answer all these questions positively.

\begin{te}\label{thm_4classes}
There exists an infinite simple sharply 2-transitive group which is finitely generated, has property (T), and has exactly four conjugacy classes.
\end{te}

%\Vcomment{Rq deplacee plus haut}
% \begin{rque}
% It is worth noting that, except for three exceptions (namely the affine groups on $\bbF_2$, $\bbF_3$ and $\bbF_4$), all finite sharply 2-transitive groups have at least five conjugacy classes. This shows that sharply 2-transitive groups with exactly four conjugacy classes are particularly exotic.\Scomment{Ajout. Pas très satisfait de cette phrase... et faut-il la laisser en remarque ou l'inclure dans le paragraphe avant le théorème ?} 
% \end{rque}

The proof of our results combines the strategy used in \cite{RT} and \cite{AT21} with small cancellation theory over relatively hyperbolic groups, as developed notably by Osin in \cite{Osi10} to construct infinite finitely generated groups with exactly two conjugacy classes.

%\newpage
\paragraph{Acknowledgements.}The first named author would like to thank Katrin Tent for stimulating discussions, and acknowledges support from the Deutsche Forschungsgemeinschaft (DFG, German Research Foundation) under Germany’s Excellence Strategy EXC 2044–390685587, Mathematics Münster: Dynamics-Geometry-Structure and from CRC 1442 Geometry: Deformations and Rigidity. The second named author benefits from the support of the French government “Investissements d’Avenir” program ANR-11-LABX-0020-01 and from the ANR project ANR-22-CE40-0004.

\section{Conjugacy classes in a sharply 2-transitive group}\label{sec_minoration}

\subsection{Characteristic of a sharply 2-transitive group}\label{characteristic}

If $G$ acts sharply 2-transitively on a set $X$ (with $\vert X\vert \geq 2$), then $G$ has involutions (take any element exchanging two points $x,y\in X$), 
and all involutions are conjugate (any involution swaps a pair of points, and two involutions can be conjugate to swap the same pair of points).

By analogy with the affine group %$GA(K)=
$K^*\semidirect K$, one says that $G\actson X$ is of characteristic $2$
if involutions have no fixed point.

%\Vcomment{Il est peut-etre vrai qu'on peut parler de la caractéristique de $G$ sans parler de $X$, mais ca ne me semble pas immediat de distinguer la caracteristique 2 d'une autre caractéristique sans parler de $X$. Comme ca peut porter a confusion plus qu'autre chose, j'aime autant parler de caracteristique de $G\actson X$ plutot que de caractéristique de $G$.}\Scomment{C'est vrai, a priori je pense qu'un groupe pourrait être S2T de carac 2 et 0, mais je n'ai pas d'exemple.}
When the action is not of characteristic $2$, every involution has a unique fixed point, and distinct involutions have distinct fixed points (see for instance \cite[Corollary 3.2]{Tent2}). 
It follows that there is a natural equivariant bijection between $X$ and the set $\cali_G$ of involutions of $G$, on which $G$ acts by conjugation. In particular, a group has a sharply 2-transitive action of characteristic distinct from 2 on a set $X$ if and only if it acts sharply 2-transitively on its set of involutions.
It also follows that all pairs of distinct involutions are conjugate.
By analogy with the affine group in characteristic $\neq 2$, we say that an element of $G$ is a \emph{translation} if it is the product of two distinct involutions. 
Thus, in characteristic $\neq 2$, all translations are conjugate, and one says that $G\actson X$
has characteristic 0 if translations have infinite order (\ie if two distinct involutions generate an infinite dihedral group)
and characteristic $p$ if translations have order $p$ (\ie any two distinct involutions generate a dihedral group of order $2p$),
and $p$ has to be some prime number (otherwise the subgroup generated by a translation $t$ of order $p$ would contain a translation $t'$ of order $p'<p$, thus not conjugate to $t$).
% Using the natural bijection between $X$ and involutions of $G$, one sees that the characteristic of $G\actson X$ depends only on $G$.
Note that if $K$ is a field, then the characteristic of $K$ in the usual sense coincides with the characteristic of $K^*\ltimes K$ as defined above.
The sharply 2-transitive groups we will construct in this paper will have characteristic 0.

%Moreover, if every involution has a fixed point (which is necessarily unique), then all translations are conjugate and we define the \emph{characteristic} of $G$ as follows: if the order of a translation is infinite then $\mathrm{char}(G):=0$, and if the order of a translation is finite then it is necessarily a prime number $p>2$ and we define $\mathrm{char}(G):=p$. If involutions act freely on $X$, we say that $G$ has characteristic 2. Note that if $K$ is a field, then the characteristic of $K$ in the usual sense coincides with the characteristic of $GA(K)$ as defined above.

\subsection{A lower bound on the number of conjugacy classes}\label{optimal_number}

A sharply 2-transitive group with exactly two conjugacy classes is necessarily abelian since all its non-trivial elements have order 2. 
But the only abelian sharply 2-transitive group is $\mathbb{Z}/2\mathbb{Z}$ (because the stabilizer of a point in such a group is malnormal), 
and hence every sharply 2-transitive group of order $>2$ has at least three conjugacy classes. 

The only finite sharply 2-transitive group with exactly three conjugacy classes is the symmetric group $S_3$, but there are infinite examples: 
Cohn constructed a skew field $K$ of characteristic 2 whose multiplicative group $K^{*}$ has exactly two conjugacy classes 
(see Corollary 2 of Theorem 6.3 in \cite{Coh71}), and the corresponding affine group $K^*\ltimes K$ has exactly three conjugacy classes.
%(indeed, the action of $k\in K^{*}$ on $K$ is defined by $k\cdot x=kx$, and therefore $K^{*}$ acts transitively on $K^{*}$). 

The goal of this section is Proposition \ref{optimal} below showing that any sharply 2-transitive group $G\actson X$ of characteristic $\neq 2$ with $|X|\geq 4$ 
has at least 4 conjugacy classes. Moreover, in characteristic 2, the only possibility to have at most 3 conjugacy classes
comes from near-fields of characteristic 2.
%the number of conjugacy classes in a sharply 2-transitive group $G$ of order $>6$ is at least four provided that $\mathrm{char}(G)\neq 2$ 
%(see Proposition \ref{optimal} below), or provided that $G$ has no non-trivial abelian normal subgroup (see Corollary \ref{optimal2}).
%which proves that the number of conjugacy classes in the group constructed in the previous subsection is optimal.
In Section \ref{sec_4conj}, we will construct infinite sharply 2-transitive groups of characteristic 0 with exactly 4 conjugacy classes.

\begin{lemme}[{\cite[Corollary 3.2]{Tent2}}]\label{lem_fix} 
  Let $G$ be a group acting sharply 2-transitively on a set $X$, with $\vert X\vert\geq 2$. Then two distinct involutions cannot have a common fixed point in $X$.

%  Equivalently, for all $x\in X$, $G_x$ contains at most one involution, 
%  and exactly one in characteristic $\neq 2$.
\end{lemme}

\begin{proof}
  The lemma is obvious if the characteristic is $2$, \ie
  if involutions don't have fixed points. Otherwise, this is \cite[Corollary 3.2]{Tent2}.
\end{proof}

\begin{lemme}\label{lem_prod_invol}
Let $G$ be a group acting sharply 2-transitively on a set $X$, with $\vert X\vert\geq 2$. The product of two distinct involutions cannot fix a point.
\end{lemme}

\begin{proof} Assume towards a contradiction that a product $uv$ of two distinct involutions 
  fixes a point $x\in X$.
  If $v$ fixes $x$, so does $u$, which contradicts the previous lemma.
  Thus the involution $v$ exchanges the point $x$ and the point $y=v\cdot x\neq x$.
  Since $uv\cdot x=x$, $u$ exchanges $x$ and $y$. By $2$-sharpness, $u=v$.
\end{proof}

\begin{prop}\label{optimal}
Let $G\actson X$ be a sharply 2-transitive group with $|X|\geq 4$. 
\begin{itemize}
    \item If $G\actson X$ is not of characteristic 2, then $G$ has at least 4 conjugacy classes.
    \item If $G\actson X$ is of characteristic 2, then $G$ has at least 3 conjugacy classes, 
    and at least 4 conjugacy classes
unless $G\simeq K^*\ltimes K$ for some near-field $K$ such that all non-trivial elements of $K^{*}$ are conjugate.
\end{itemize}
\end{prop}

% \begin{rque}
% As explained before, Cohn constructed in \cite{Coh71} an infinite non-commutative field $K$ of characteristic 2 such that all elements of $K\setminus \lbrace 0,1\rbrace$ are conjugate in the multiplicative group $K^{*}$ of $K$, and hence $K^{*}$ has exactly two conjugacy classes and $GA(K)=K\rtimes K^{*}$ has exactly three conjugacy classes. More generally, for every commutative field $k$, Cohn constructed a non-commutative field $K\supset k$ such that the center of $K^{*}$ coincides with $k^{*}$ and all elements in $K^{*}\setminus k^{*}$ are conjugate (we refer the reader to the discussion following Corollary 2 of Theorem 6.3 in \cite{Coh71}). Taking for $k$ the finite field $\mathbb{F}_q$, we get a non-commutative field $K\supset \mathbb{F}_q$ such that $K^{*}$ has exactly $q$ conjugacy classes. Therefore, $GA(K)$ has exactly $q+1$ conjugacy classes.\end{rque}

%The following corollary of Proposition \ref{optimal} is immediate.

%\begin{co}\label{optimal2}\Vcomment{Je suis pas sûr de voir l'interet du corollaire par rapport a la prop}Let $G$ be a sharply 2-transitive group. If $G$ has no non-trivial abelian normal subgroup, then $G$ has at least four conjugacy classes.\end{co}

%We now prove Proposition \ref{optimal}.

\begin{proof}
We already saw at the beginning of this section that if $|X|\geq 3$, $G$ has at least 3 conjugacy classes.

Fix a point $x\in X$. By Lemmas \ref{lem_fix} and \ref{lem_prod_invol}, $G_x$ contains at most one involution and 
no translation (\ie a product of two distinct involutions). 
Moreover, $\vert X\setminus \lbrace x\rbrace\vert \geq 3$
and $G_x$ acts transitively on $X\setminus \lbrace x\rbrace$.
Thus, $G_x$ contains a non-trivial element $g_x$ which is neither an involution nor a translation.

Assume $G\actson X$ has characteristic $p\neq 2$, and consider two distinct involutions $u,v$.
Then the translation $uv$ is an element of order $p$ (or infinite order if $p=0$), hence is not an involution.
This shows that the elements $1,u,uv,g_x$ are in four distinct conjugacy classes.

In characteristic 2, the argument above says that if some translation $uv$ in not an involution
then $G$ has at least four conjugacy classes. 
We may thus assume that the product of any two distinct involutions is an involution, \ie that the set $\cali_G\cup\{1\}$ is a (normal) subgroup of $G$.
By \cite[Theorem 7.1]{Kerby}, $G=K^*\ltimes K$ for some near-field $K$ (where $K=\cali_G\cup\{1\}$ as an additive group). If $G$ has at most 3 conjugacy classes, then $K^*$ has at most 2 conjugacy classes.
%By \cite[Lemma 11.47]{BN}, we have $G=K\rtimes G_x$ where $G_x$ acts freely and transitively by conjugation on $K\setminus \lbrace 1\rbrace=\cali_G$. It can be proved that $K$ can be endowed with a near-field structure and that $G$ is isomorphic to the affine group $GA(K)=K\rtimes K^{*}$ (see for instance \cite{BN}, page 225). 
%If $G$ has $<4$ conjugacy classes, then $K^{*}$ has $<3$ conjugacy classes, and $K^{*}$ has $>1$ conjugacy classes since $G$ is not isomorphic to $S_2$ by assumption, and therefore $K^{*}$ has exactly two conjugacy classes.
\end{proof}

\section{The classes of groups $\calc$ and $\calc'$}

In this section that closely follows \cite{AT21,RT}, 
we introduce two classes of groups $\calc$ and $\calc'$
that are stable under various types of HNN extensions, amalgams and increasing union.
Any sharply 2-transitive group of characteristic 0 is in the class $\calc$, but the class $\calc'$ is more restrictive.
We show that these classes are stable under

\subsection{The class $\calc$}
 In \cite{AT21} (building on \cite{RT}), a class of groups $\mathcal{C}$ is introduced
that contains all sharply 2-transitive groups of characteristic 0 and which is preserved under various HNN extensions and amalgamated products. 
We reformulate this definition below (see Definition \ref{dfn_class_C}).

We denote by $D_\bbZ=\bbZ/2\bbZ\semidirect \bbZ$ the infinite dihedral group and by
$D_\bbQ=\bbZ/2\bbZ\semidirect \bbQ$ the group of isometries of $\bbQ$.

A subgroup $H\subset G$ is called \emph{malnormal} if the intersection $H\cap gH g^{-1}$ is trivial for every $g\in G\setminus H$, and \emph{quasi-malnormal} if $H\cap gH g^{-1}$ has order at most 2 for $g\in G\setminus H$. 

%If $X$ is a set, we denote by $X^{(2)}=\{(x,x')\in X^2 \ | \ x\neq x'\}$ the set of ordered pairs of elements of $X$.
Recall that we denote by $\cali_G$ the set of involutions of 
$G$, and we denote by $\cali_G^{(2)}=\{(u,v)\in \cali_G^2 \ | \ u\neq v\}$ 
the set of ordered pairs of involutions, on which $G$ acts by conjugation. 

\begin{de}\label{dfn_type} Let $G$ be a group, and $(u,v)\in\cali_G^{(2)}$ be a pair of distinct involutions.
  \begin{itemize}
  \item   We say that $(u,v)$ is of type $D_\bbZ$ if $\grp{u,v}$ is contained in
    a (necessarily unique) group $D_{u,v}$ isomorphic to $D_\bbZ$ with $D_{u,v}$ quasi-malnormal in $G$.
    If $\grp{u,v}=D_{u,v}$ we say that the pair $(u,v)$ is \emph{maximal}.
  \item We say that $(u,v)$ is of type $D_\bbQ$ if $\grp{u,v}$ is contained in a (maybe non-unique, maybe not quasi-malnormal) group isomorphic to $D_\bbQ$.
\end{itemize}
\end{de}

\begin{rque}\label{rem_type}\ 
  \begin{enumerate}
  \item If $(u,v)$ is of one of the two types, then $\grp{u,v}$ is infinite dihedral.

  \item \label{it_malnormal} If $(u,v)$ is of type $D_\bbZ$ and $(u',v')$ is of type $D_\bbQ$, then $\grp{u,v}\cap \grp{u',v'}$ has cardinality at most 2. In particular, the two types are mutually exclusive.
    
\item \label{it_commensurator} Recall that the \emph{commensurator} $\Comm_G(H)$ of a subgroup $H\subset G$ is the subgroup consisting of elements $g\in G$ such that the intersection of $H$ and $gHg^{-1}$ has finite index in both of them. The pair $(u,v)$ is of type $D_\bbZ$ if and only if the commensurator $\Comm_G(\grp{u,v})$ is an infinite dihedral group. In this case, $D_{u,v}=\Comm_G(\grp{u,v})$ is the unique maximal infinite dihedral group containing $\grp{u,v}$.
  
%\item \label{it_commensurator2} If $(u,v)$ is of type $D_\bbQ$, then the commensurator of $\grp{u,v}$ contains a group isomorphic to $D_\bbQ$.
  \end{enumerate}
\end{rque}

\begin{de}\label{dfn_class_C}
Let $G$ be a group. 
  We say that $G$ belongs to the class $\mathcal{C}$ if 
\begin{enumerate}
\item  \label{it_IG} $G$ acts transitively on $\cali_G$, and freely on $\cali_G^{(2)}$,
%\item any two distinct involutions of $G$ generate an infinite dihedral group.
%Equivalently, all involutions of $G$ are conjugate and the centralizer of any pair of distinct involutions is trivial
\item \label{it_uv} any pair $(u,v)\in \cali_G^{(2)}$ of distinct involutions is of type $D_\bbZ$ or $D_\bbQ$,
\item \label{it_DQ} the set of pairs $(u,v)\in \cali_G^{(2)}$ of type $D_\bbQ$ is non-empty and $G$ acts transitively on it.
\end{enumerate}
\end{de}
% For example, the affine group $GA(\bbQ)=\bbQ^*\ltimes \bbQ$ is in the class $\calc$ with all pairs of involutions of type $D_\bbQ$:
% the set of all involutions is contained in the subgroup of isometries $\{\pm 1\}\ltimes \bbQ\simeq D_\bbQ$
% and is in one-to-one correspondence with the affine line $\bbQ$.
% The action of $GA(\bbQ)$ on the affine line being sharply 2-transitive, so is its action on the set of involutions.

\begin{rque} \label{rq_C_et_s2t}
    If a group $G$ belongs to $\calc$, and if all pairs of involutions of $G$ are of type $D_\bbQ$, then $G$ acts
    sharply 2-transitively on $\cali_G$.

Conversely, any sharply 2-transitive group $G$ of characteristic 0 belongs to the class $\calc$, with all pairs of involutions of type $D_\bbQ$.
  Indeed, $G$ acts sharply 2-transitively on its set of involutions (as recalled in Subsection \ref{characteristic}). 
It then suffices to show that any pair of involutions $(u,v)\in \cali_G^{(2)}$ is of type $D_\bbQ$.
  % Given a pair of involutions $(u,v)\in \cali_G^{(2)}$  and
  Indeed, given $k\geq 1$,
  $(u,v)$ is thus conjugate to $(u,(vu)^{k-1}v)$, so 
  consider an inner automorphism $\sigma_k=\ad_{a_k}$ sending
    $(u,v)$ to $(u,(vu)^{k-1}v)$. Note that $\sigma_k(\grp{u,v})$ is the unique subgroup of index $k$ of $\grp{u,v}$.
    Moreover, for all $k,l\in \bbN\setminus\{0\}$, 
    $\sigma_{k}\circ\sigma_l$ agrees with $\sigma_{kl}$ in restriction to $\grp{u,v}$.
    Since the pair $(u,v)$ has trivial centralizer, it follows that $a_ka_l=a_{kl}$ so 
    $\sigma_k$ and $\sigma_l$ commute. Define $D_1=\grp{u,v}$ and $D_n=\sigma_{n!}\m(D_1)=\sigma_{n}\m (D_{n-1})$.
    Since $\grp{u,v}$ contains $\sigma_n(\grp{u,v})$ with index $n$,
    $D_n=\sigma_{n!}\m(\grp{u,v})$ contains $D_{n-1}=\sigma_{n!}\m\sigma_{n}(\grp{u,v})$ with index $n$. 
    %Note that $D_{n-1}=\ad_{a_n}(D_n)$. 
    Thus, $D_1\subset D_2\subset \dots \subset D_n$ with $[D_n:D_{n-1}]=n$.
    It easily follows that $\cup_{n\in\mathbb{N}} D_n$ is an increasing union of infinite dihedral groups,
  each one having index $n$ in the next one, and whose union is isomorphic to $D_\bbQ$.
\end{rque}

\begin{rque}\label{rem_cent}
  Assertion 1 implies that the centralizer of each involution is malnormal. Indeed, if $g\notin \Cent_G(i)$, any element $a\in \Cent_G(i)^g\cap\Cent_G(i)$ fixes two distinct involutions hence is trivial. Assertion 2 implies that any two distinct involutions generate an infinite dihedral group, and thus do not commute.
\end{rque}

\begin{rque}Definition \ref{dfn_class_C} is equivalent to the definition of the class $\calc$ that appears in \cite{AT21}. Since we do not use the equivalence of the two definitions in this paper, we leave the verification of the equivalence to the interested reader.
\end{rque}

\subsection{Three types of elements}

The following terminology will be convenient.
\begin{de}\label{dfn_typeh}
Consider a group $G$ and an element $h\in G$ of infinite order. We say that
  \begin{enumerate}
\item $h$ is a \emph{translation} if it is the product of two distinct involutions,
\begin{itemize}
    \item it is a \emph{translation of type $D_\bbQ$} if it is the product of two distinct involutions $h=uv$, with $\grp{u,v}$ of type $D_\bbQ$,
   \item it is a \emph{translation of type $D_\bbZ$} if it is the product of two distinct involutions $h=uv$, with $(u,v)$ of type $D_\bbZ$;
\end{itemize}
  \item $h$ is a \emph{homothety} if $h$ centralizes some involution;
  \item $h$ is \emph{isolated} if $h$ is contained in a malnormal cyclic group $\grp{\hat h}$. We say that $h$ is \emph{maximal isolated}
    if $\grp{h}$ is malnormal.
  \end{enumerate}
\end{de}

\begin{rque}In any group $G$, the two types of translations are mutually exclusive (by the second point of Remark \ref{rem_type}). 
Moreover, it is clear that an isolated element cannot be a translation or a homothety. 
\end{rque}

\begin{rque}
    The set of translations together with the trivial element does \emph{not} form a subgroup of $G$ in general. This is the case for $G=K^*\ltimes K$ when $K$ is a skew field, and in fact, if and only if $G=K^*\ltimes K$ for some near-field $K$, where $K=\cali_G^2$ as an additive group (see \cite[Theorem 7.1]{Kerby}).
\end{rque}

\begin{lemme}
  In a group belonging to $\calc$, the three possibilities for $h$ (translation, homothety, isolated) are mutually exclusive.
\end{lemme}

%_{Referer au lemme \ref{lem_prod_invol} a enoncer pour action  2-sharp (pas 2-transitive) ?}
\begin{proof}%_{preuve legerement reformulee}Indeed, consider the commensurator $C=\Comm(\grp{h})$. Then $C=\grp{\hat h}$ for $h$ isolated, $C=D_{u,v}\simeq D_\bbZ$ for a $D_\bbZ$-translation, and $C$ contains a group isomorphic to $D_\bbQ$ for a $D_\bbQ$-translation. Thus 1,2 and 4 are mutually exclusive. If $h$ is a homothety centralizing an involution $i$, then $i\in C$ so $h$ cannot be isolated.
  
  In view of the previous remark, we just have to prove that a 
  homothety $h$ cannot be a translation. %Let $h$ be a homothety. 
  Let $i$ be an involution centralized by $h$.
  Arguing by contradiction, assume that $h=uv$ for two involutions $u\neq v$.
%  We finally check that $h$ cannot be a translation using $2$-sharpness of the action on involutions.
  We will use several times that two distinct involutions cannot commute in a group of class $\calc$. 
We note that $u,v\neq i$: indeed,
  if $u=i$ for instance, then $v$ has to commute with $i$ so $v=i$ and $h=1$ a contradiction.
  % Since no two distinct involutions commute,
  The involution $j=uiu\m$ is therefore distinct from $i$. The conjugation by $u$ exchanges $i$ and $j$
  and so does $v$ since $uv$ commutes with $i$.
  By $2$-sharpness, $u=v$ a contradiction.
\end{proof}

\begin{rque}
If $G$ is in the class $\calc$ then all involutions are conjugate, and
  all translations of type $D_\bbQ$ are also conjugate (because $G$ acts transitively on pairs of involutions of type $D_\bbQ$).
\end{rque}

\begin{de}\label{3-type}
  We say that a group $G$ satisfies the \emph{3-type condition} if its elements are of order $1,2$, or $\infty$, and each element $h\in G$ of infinite order is 
  a translation, a homothety or is isolated.
%  of one of the four types of Definition \ref{dfn_typeh}.
\end{de}

\subsection{The class $\calc'$}
We will use the following subclass $\calc'$ of $\calc$.

\begin{de}[The class $\calc'$]
  We say that a group $G$ belongs to the class $\calc'$ if it lies in the class $\calc$ and satisfies the 3-type condition.
\end{de}

For example, the affine group $\bbQ^*\ltimes\bbQ$ belongs to $\calc'$ with no isolated element and no $D_\bbZ$-translation. 
More generally, if $K$ is a field of characteristic 0, then $K^*\ltimes K$ 
belongs to $\calc'$ if and only if
%does not belong to $\calc'$ if 
$K^*$ contains no primitive $n$-th root of unity with $n>2$. 

The following lemmas give some restrictions for groups in $\calc'$ that will be useful.

\begin{lemme}\label{lem_fini}
If $G$ is a group in which every finite subgroup of $G$ has order at most 2, then every infinite virtually cyclic subgroup of $G$ is isomorphic to $\bbZ$, $\bbZ\times (\bbZ/2\bbZ)$ or $D_\bbZ$.\end{lemme}

\begin{proof}
If $E\subset G$ is virtually cyclic, it can be written as an extension of a finite normal group $N\subset E$,
with $E/N$ isomorphic to $\bbZ$ or $D_\bbZ$. If $N$ is trivial, we are done. Otherwise, $N\simeq\bbZ/2\bbZ$
and since finite subgroups of $G$ have order $\leq 2$, $E/N$ has to be torsion-free, so $E/N\simeq \bbZ$ and $E\simeq \bbZ\times (\bbZ/2\bbZ)$.
\end{proof}

\begin{lemme}\label{lem_fini2}
If $G$ belongs to $\calc'$, then every non-trivial finite subgroup of $G$ has order 2, and every infinite virtually cyclic subgroup of $G$ is isomorphic to $\bbZ$, $\bbZ\times (\bbZ/2\bbZ)$ or $D_\bbZ$.
\end{lemme}

\begin{proof}
If $G$ belongs to $\calc'$, the order of elements of finite order in $G$ is at most 2, so any finite subgroup $F\subset G$ is commutative.
Since distinct involutions of $G$ never commute, $F$ is trivial or isomorphic to $\bbZ/2\bbZ$.
The second part of the lemma then follows from Lemma \ref{lem_fini}.
\end{proof}

\begin{lemme}\label{lemme_exists_homothetie}
Any group $G$ belonging to $\calc'$ contains a subgroup isomorphic to $\bbZ\times(\bbZ/2\bbZ)$.
\end{lemme}

\begin{proof}
By Definition \ref{dfn_class_C} defining class $\calc$, 
$G$ contains a pair of involutions $(u,v)$ of type $D_\bbQ$, and all such pairs are in the same $G$-orbit. 
Moreover, there are infinitely pairs $(u,v')$ of type $D_\bbQ$ (take for $v'$
any involution in $\grp{u,v}\setminus\{u\}$). Since the centralizer $Z(u)$ of $u$ acts transitively on all such $v'$,
it is infinite. Since $u$ is the only involution in $Z(u)$, it contains an element of infinite order which proves the lemma.
\end{proof}

\subsection{Increasing unions}

Recall from Definition \ref{dfn_type} that a pair of involutions $(u,v)$ of type $D_\bbZ$ is maximal if $\grp{u,v}$ is quasi-malnormal in $G$. 
\begin{de}
    Consider an embedding of groups $G\subset G'$, with $G,G'$ in $\calc$.  
  \begin{itemize}
  \item We say that the embedding \emph{preserves maximal pairs} if for each maximal pair $(u,v)\in \cali_{G}^{(2)}$ of
    type $D_\bbZ$, either the pair $(u,v)$ is still a maximal pair of type
    $D_\bbZ$ in $G'$, or $(u,v)$ is of type $D_\bbQ$ in $G'$.

  \item We say that the embedding \emph{preserves maximal isolated elements} if for each malnormal cyclic group $\grp{h}$, either $\grp{h}$ is still
  malnormal in $G'$, or $h$ is a translation or a homothety in $G'$.
\end{itemize}
We say that the embedding \emph{preserves maximality} if it satisfies both conditions.
\end{de}

\begin{lemme}\label{lem_chaine}
  Consider $G=\cup_{n\in\bbN} G_n$ an increasing union of groups in $\calc'$. Assume that the embeddings $G_n\subset G_{n+1}$ preserve maximality. Then $G$ is in the class $\calc'$ and for each $n\in \bbN$, the embedding $G_n\subset G$ preserves maximality.
\end{lemme}

\begin{proof}
  Let's first check that $G$ is in the class $\calc$.
  Assertion (1) of Definition \ref{dfn_class_C} is clear. If $(u,v)$ is of type $D_\bbQ$ in $G_n$, then it is clearly of type $D_\bbQ$ in $G$.
  If $(u,v)\in \cali_{G_{n_0}}^{(2)}$ is of type $D_\bbZ$ in every $G_{n}$ for all $n\geq n_0$, 
  then since embeddings preserve maximality,  
  $D_{u,v}$ does not depend on the group $G_n$.
  It follows that $D_{u,v}$ is quasi-malnormal in every $G_n$ hence in $G$. This proves Assertion (2) and that the embedding $G_n\subset G$ preserves maximal pairs. 
  Since any pair of type $D_{\bbQ}$ in $G$ is of type $D_{\bbQ}$ in any $G_n$ containing it, Assertion (3) is clear, so $G$ belongs to $\calc$.

  To prove that $G$ is in $\calc'$, the only non-obvious point to check is that if $h\in G_{n_0}$ is isolated in $G_n$ for every $n\geq n_0$ then it is still isolated in $G$. The argument is similar to the previous one 
  using that inclusions preserve maximal isolated elements.
\end{proof}

\subsection{HNN extensions and amalgams}
The following proposition unifies several constructions in \cite{AT21} and \cite{RT}.
% show that the class $\mathcal{C}$ is stable under certain HNN extensions. The first one is Proposition 4.1 in \cite{AT21} (see also Proposition 1.4 in \cite{RT}).

\begin{prop}\label{prop_gen}
  Let $G$ be a group in the class $\calc$. Consider an HNN extension $G_1=G*_C$ or an amalgam $G_1=G*_C H$
  and denote by $G_1\actson T$ the corresponding Bass-Serre tree.
  %be the fundamental group of a graph of groups
  % with a vertex $v_0$ whose vertex group is $G_{v_0}=G$. Assume that
  Assume that the following properties hold for some $k\geq 0$:
  \begin{enumerate}
  \item almost $k$-acylindricity: the stabilizer of any segment of length $>k$ in $T$ has order $\leq 2$;
  \item \label{it_H} in the case of an amalgam, $H$ has at most one involution (necessarily central in $H$), and this involution lies in the subgroup $C$.
  \end{enumerate}
  Then $G_1$ is in the class $\calc$. 
  
  If moreover $G$ and $H$ satisfy the 3-type condition (see Definition \ref{3-type}), then $G_1$ belongs to $\calc'$ and the embedding $G\subset G_1$ preserves maximality.
\end{prop}

\begin{ex}\label{ex_produit_libre}
The proposition implies that if $G$ lies in $\calc$ or $\calc'$, then so does $G*\bbZ$.
\end{ex}

\begin{rque}\label{rem_H}If $H$ has a unique involution, then all elements of infinite order of $H$ are homotheties, and $H$ satisfies the 3-type condition if and only if its elements are of order $1,2$ or $\infty$.
\end{rque}

\begin{proof}
  Before starting the proof, given two involutions $u,v$, we say that $\grp{u,v}$ is hyperbolic if it does not fix a point in $T$ (otherwise it is elliptic).
  In this case, we denote by $l_{u,v}$ the unique line of $T$ invariant under $\grp{u,v}$. 
  By acylindricity, the pointwise stabilizer of $l_{u,v}$ has order $\leq 2$, we claim that it is trivial.
  Indeed, if $g\in G$ is an involution fixing $l_{u,v}$ pointwise, it is unique and therefore has to commute with $\grp{u,v}$.
  Since $g$ has a common fixed point with $u$, the pair $(g,u)$ is conjugate in $G$ (it cannot be conjugate in $H$ because $g\neq u$).
  Since no pair of involutions of $G$ commute, we get a contradiction which proves the claim.

  We now check that $G_1$ satisfies Assertion (1) of Definition \ref{dfn_class_C}. The fact that $G_1$ acts transitively on the set of its involutions is clear: any involution of $G_1$ is elliptic in $T$, so it is conjugate in $G$ or $H$,
  hence in $G$ by assumption (\ref{it_H}). Then, proving that the action on $\cali_{G_1}^{(2)}$ is free amounts to checking that the centralizer $Z$ of a pair $(u,v)\in  \cali_{G_1}^{(2)}$ is trivial. If $\grp{u,v}$ is hyperbolic, then any element $g\in Z$ fixes $l_{u,v}$ pointwise so $g=1$ by the initial claim.
  If $\grp{u,v}$ is elliptic, then the pair $(u,v)$ is conjugate to a pair of involutions of $G$ (it cannot be conjugate in $H$ by (\ref{it_H})). The set $F\subset T$ of fixed points of $\grp{u,v}$ is a subtree of diameter $\leq k$.
  In the case of an amalgam, it is reduced to a point fixed by a conjugate of $G$ by Assumption (\ref{it_H}).
  Since $Z$ preserves the bounded tree $F$, it fixes a point in $F$. Thus $Z,u,v$ are contained in a common conjugate of $G$, so $Z$ is trivial since $G$ belongs to $\calc$.
  This proves that the action of $G_1$ on $\cali_{G_1}^{(2)}$ is free.
  
  Then we prove that $G_1$ satisfies Assertion (2) of Definition \ref{dfn_class_C}, namely that any pair of distinct involutions is of type $D_{\bbZ}$ or $D_{\bbQ}$. If $\grp{u,v}$ is hyperbolic then the pair $(u,v)$ is of type $D_\bbZ$: the global stabilizer $D_{u,v}$ of  $l_{u,v}$ of is an infinite dihedral group (it acts faithfully on $l_{u,v}$ by our initial claim),
  and it is quasi-malnormal: if $gl_{u,v}\neq l_{u,v}$, then $D_{u,v}\cap D_{u,v}^g$ cannot contain any element of infinite order, so $D_{u,v}\cap D_{u,v}^g$ has order at most $2$. If $\grp{u,v}$ is elliptic, then its commensurator $M=\Comm_{G_1}(\grp{u,v})$ is elliptic because it preserves the union of fixed points of finite index subgroups of $\grp{u,v}$, a subtree of diameter $\leq k$.
  Up to conjugacy, we may thus assume that $\grp{u,v}\subset M\subset G$, so $M=\Comm_{G}(\grp{u,v})$.
  If $(u,v)$ has type $D_\bbQ$ in $G$, it also has type $D_\bbQ$ in $G_1$. If it has type $D_\bbZ$ in $G$, then $\Comm_G(\grp{u,v})$ is infinite dihedral, and so is $\Comm_{G_1}(\grp{u,v})$
  so $(u,v)$ has type $D_\bbZ$ in $G_1$. Hence Assertion (2) of Definition \ref{dfn_class_C} is proved. Note that we also proved that the embedding $G\subset G_1$ preserves maximal pairs.

    Last, since any pair of type $D_\bbQ$ has to be elliptic in $T$, it is conjugate to a pair in $G$, so $G_1$ acts transitively
  on the set of pairs of type $D_\bbQ$. This shows that $G_1$ satisfies Assertion (3) of Definition \ref{dfn_class_C} and concludes the proof that $G_1$ belongs to $\calc$.

  We now assume that $G$ and $H$ satisfy the 3-type condition and prove that $G_1$ satisfies the 3-type condition, hence belongs to the class $\calc'$. Clearly, all elements of $G_1$ have order $1,2$ or $\infty$. Let $h\in G_1$ be an element of infinite order.

  If $h$ is hyperbolic in $T$, let $l\subset T$ be its axis,
  and $A$ be the global stabilizer of $l$. If the pointwise stabilizer of $l$ is non-trivial,
  it is isomorphic to $\bbZ/2\bbZ$ by acylindricity and $h$ is a homothety.
  If the pointwise stabilizer of $l$ is trivial, then $A$ acts faithfully on $l$ so $A$ is either cyclic
  or infinite dihedral. If dihedral, then $h$ is a translation.
  If $A$ is cyclic, then it is malnormal because if $gAg\m\cap A\neq 1$, then $g$ must preserve the line $l$,
  so $g\in A$. Thus $h$ is isolated with $\grp{\hat h}=A$.

  Now assume that $h$ is elliptic in $T$. 
  Consider the subtree $F$ of $T$ consisting of points fixed
  by some power of $h$. $F$ has diameter at most $k$ and is invariant under the group $M=\Comm_{G_1}(\grp{h})$
  so $M$ is elliptic. Thus we may assume that $\grp{h}\subset M\subset H$ or $\grp{h}\subset M\subset G$.
     
  If $\grp{h}\subset M\subset H$ and $H$ centralizes an involution, then $h$ is a homothety.
  If $\grp{h}\subset M\subset H$ and $H$ is torsion-free, then $h$ cannot be a translation or a homothety in $H$
  so $h$ is isolated in $H$, so $M\simeq \bbZ$ and $h$ is isolated in $G_1$.

    If $\grp{h}\subset M\subset G$, and if $h$ is a translation or a homothety in $G$, then this is also the case in $G_1$.
    If $h$ is isolated in $G$, then $M\simeq \bbZ$ and $h$ is isolated in $G_1$.
%   So we may assume that $h$ is isolated in $G$ 
%   \Vcomment{Raccourcir: il suffit de dire que $h$ isolé ssi son commensurateur est cyclique, non ?}
%   and consider $\hat h\in G$ such that $\grp{\hat h}$ is malnormal in $G$.
%   Then $C=\grp{\hat h }$ and it follows that $\grp{\hat h}$ is malnormal in $G_1$: if
%     $\grp{\hat h}^g\cap\grp{\hat h}\neq \{1\}$ then $g$ commensurates $\grp{h}$ and $g\in C$.
% This proves that $\grp{h}$ is malnormal in $G_1$ so $h$ is isolated in $G_1$, which concludes the proof that $G_1$ belongs to the class $\calc'$. Note that we also proved that $G\subset G_1$ preserves maximal isolated elements, hence preserves maximality.
\end{proof}

The following corollary allows to turn a pair of involutions of type $D_\bbZ$ into type $D_\bbQ$.

\begin{co}[{\cite[Prop. 4.1]{AT21}, \cite[Prop 1.4]{RT}}]\label{cor_HNN_paire}
  Let $G$ be a group in the class $\mathcal{C}'$.
  Let $(u_0,v_0)\in  \cali_G^{(2)}$ be a pair of type $D_\bbQ$ and $(u,v)\in \cali_G^{(2)}$ be any maximal pair of type $D_\bbZ$. Then the following HNN extension belongs to $\calc'$:\[G_1=\langle G,t \ \vert \ tut^{-1}=u_0, \ tvt^{-1}=v_0\rangle.\]Moreover, the inclusion $G\subset G_1$ preserves maximality.
\end{co}

\begin{proof}
  In view of Proposition \ref{prop_gen}, we just have to check that the action on the Bass-Serre tree is almost $2$-acylindrical.
  This easily follows from the fact that $\grp{u,v}$ is quasi-malnormal in $G$, and that $|\grp{u_0,v_0}\cap \grp{u,v}^g|\leq 2$
  for all $g\in G$.
\end{proof}

The following corollary allows to turn an isolated element into a translation.

\begin{co}[see also Proposition 4.2 in \cite{AT21}]\label{cor_HNN_isolated}
  Let $G$ be a group in $\calc'$ and let $g,h$ be two elements of $G$ of infinite order. Suppose that $g$ is a translation or a homothety in $G$ and that 
  $h$ is a maximal isolated element. 
  Then the following HNN extension belongs to $\calc'$: \[G_1=\langle G,t \ \vert \ tgt^{-1}=h\rangle.\]Moreover, the inclusion $G\subset G_1$ preserves maximality.
\end{co}

%\begin{rque}This result is also true if we assume that $g$ is a homothety, but we do not need it.\end{rque}

\begin{proof}
  Acylindricity follows from malnormality of $\grp{h}$ in $G$, and
  from the fact that $\grp{h}\cap \grp{g}^a=\{1\}$ for all $a\in G$.
  This last fact holds true because otherwise,
  some power of $h$ would be conjugate in a group isomorphic to $D_\bbZ$ or to $\bbZ\times (\bbZ/2\bbZ)$
  contradicting the fact that $\grp{h}$ is malnormal in $G$.
\end{proof}

The following result allows to make two given homotheties conjugate.

\begin{co}\label{cor_HNN_homothetie}
  Let $G$ be a group in the class $\calc'$.
  Let $h_1,h_2\in G$ be two homotheties % elements of the same order that centralize an invoand commuting with the same involution $i$.
  centralizing respectively the involutions $u_1$ and $u_2$. Then the following HNN extension belongs to $\calc'$: \[G_1=\langle G,t \ \vert \ tu_2t^{-1}=u_1, \ th_2t^{-1}=h_1\rangle.\] Moreover, the inclusion $G\subset G_1$ preserves maximality.
\end{co}

\begin{proof}
  This HNN extension is not acylindrical. We rewrite it as an acylindrical amalgam as follows.
  First, since all involutions of $G$ are conjugate, we may replace $h_2$ and $u_2$ by some conjugates 
  and change $t$ accordingly
  to ensure that $u_2=u_1$, and we let $u=u_1=u_2$. Then  
  $G_1=\grp{ G,t\ |\ t u t\m= u,\ th_2t\m=h_1}$.
  Let $Z=\Cent_G(u)$ be the centralizer of $u$ in $G$.
  Then $G_1= G*_Z H$ where $H$ is the HNN extension $$H=\grp{Z,t\ \vert \ tut\m=u,\ th_2t\m =h_1}.$$
  Since no two involutions commute in $G$, $u$ is the unique involution of $Z$.
  Since $H$ is an HNN extension, any involution of $H$ is conjugate to an element of $Z$ thus conjugate to $u$.
  On the other hand, we see that $u$ is central in $H$, so $u$ is the unique involution of $H$.
  Moreover, all elements of $H$ are of order $1$, $2$ or $\infty$ because this is the case for $Z\subset G$.
  As noted in Remark \ref{rem_H}, this shows that $H$ satisfies the 3-type condition.

  Now $Z$ is malnormal in $G$ (see Remark \ref{rem_cent}) and it follows that the Bass-Serre tree of the amalgam $G_1= G*_Z H$
  is $2$-acylindrical. Thus Proposition \ref{prop_gen} applies.
\end{proof}

\section{An infinite simple sharply 2-transitive group with exactly four conjugacy classes}\label{sec_4conj}

%In the first subsection below, we will 
In this section, we use the tools of the previous section to 
construct an infinite countable simple sharply 2-transitive group of characteristic 0 in $\calc'$ with exactly four conjugacy classes (but still not finitely generated).
As proved in Proposition \ref{optimal}, this number of conjugacy classes is the smallest possible.
%and in the second subsection we will prove that this number of conjugacy classes is as small as possible.

%We will then use this group to build a finitely generated one in Section \ref{sec_fg}.

%In the second subsection, we will prove that a sharply 2-transitive group of order $>6$ has at least four conjugacy classes provided that $\mathrm{char}(G)\neq 2$ (see Proposition \ref{optimal} below), or provided that $G$ has no non-trivial abelian normal subgroup (see Corollary \ref{optimal2}), which proves that the number of conjugacy classes in 

%\subsection{Construction}

\begin{te}\label{4_classes}
Given a countable group $G_0$ belonging to $\calc'$,
there exists a countable group $G$ containing $G_0$ such that 
\begin{itemize}
    \item $G$ is sharply 2-transitive and belongs to $\calc'$,
    \item $G$ has exactly four conjugacy classes: 
the trivial element, the set of involutions, the set of translations and the set of homotheties. 
\end{itemize}
\end{te}

\begin{proof}
  Starting from $G_0$, we are going to construct inductively
  an increasing sequence of groups $(G_n)_{n\in\mathbb{N}}$ in the class $\mathcal{C'}$ 
  such that the group $G=\cup_{n\in\mathbb{N}}G_n$ is in the class $\calc'$, and such that all its pairs of involutions are of type $D_\bbQ$.
    As noticed in Remark \ref{rq_C_et_s2t}, this implies that $G$ acts sharply 2-transitively on its set of involutions.  
%  a sharply 2-transitive group of characteristic 0 with exactly four conjugacy classes, namely the trivial element, the set of involutions, the set of $D_\bbQ$-translations, and the set of homotheties.

  We note that if $g\in G_n$ is a homothety (\ie it centralizes an involution) in $G_n$, then it still is a homothety in $G_{n+1}$.
  Similarly, if $g\in G_n$ is a translation (\ie a product of 2 distinct involutions) in $G_n$, it still is a translation in $G_{n+1}$,
   its type may change from $D_\bbZ$ to $D_\bbQ$ but not the other way around.
  
%  We start with $G_0=GA(\mathbb{Q})=\bbQ^*\ltimes \bbQ$. We already noted that it is in the class $\calc'$ (with no isolated element and no pair of involutions of type $D_\bbZ$). 
  
  Suppose that the group $G_n$ has already been constructed, and let's construct $G_{n+1}$  as an increasing union of groups $G^m_n$, starting
  with $G_n^0=G_n$.
  We fix $(u_0,v_0)$ a pair of involutions of type $D_\bbQ$ in $G_0$.
  Consider an enumeration $g_1,g_2,\dots,$ of the maximal isolated elements of $G_n$,
  an enumeration $h_1,h_2,\ldots$ of its homotheties,
  and an enumeration $(u_1,v_1), (u_2,v_2),\dots$ of the set maximal pairs of involutions of type $D_\bbZ$.
  For each homothety $h_k$, we denote by $i_k$ the unique involution it centralizes (uniqueness follows from 2-sharpness of the action on involutions).
  Starting with $G_n^0=G_n$, we define inductively an increasing sequence of groups $G_n^m$ as follows.
\begin{enumerate}
\item  Define $G_n^{m+1/3}=\grp{G_n^{m},t \ | \ tu_m t\m=u_0, \ tv_mt\m=v_0}$ as in Corollary \ref{cor_HNN_paire} if $(u_m,v_m)$ is maximal of type $D_\bbZ$  in $G_n^{m}$ and
  $G_n^{m+1/3}=G_n^{m}$ otherwise.
\item Define $G_n^{m+2/3}=\grp{G_n^{m+1/3},t \ | \ tg_mt\m=u_0v_0}$ as in Corollary \ref{cor_HNN_isolated} if $g_m$ is a maximal isolated element in $G_n^{m+1/3}$, and $G_n^{m+2/3}=G_{n}^{m+1/3}$ otherwise.
  \item Define $G_n^{m+1}=\grp{G_n^{m+2/3},t \ | \ ti_mt\m=i_1,\ th_m t\m= h_1}$ as in Corollary \ref{cor_HNN_homothetie};
\end{enumerate}
We finally define $G_{n+1}=\bigcup_{m}G_n^m$.
Corollaries \ref{cor_HNN_paire}, \ref{cor_HNN_isolated} and \ref{cor_HNN_homothetie} 
show that each group $G_n^m$ is in $\calc'$ and that the embedding of each group in the next one preserves maximality. 
Lemma \ref{lem_chaine} then concludes that $G_{n+1}$ belongs to $\calc'$, and that the embedding $G_n\subset G_{n+1}$ preserves maximality.
Applying again Lemma \ref{lem_chaine} to $G=\cup_n G_n$, we see that $G$ belongs to $\calc'$.

We now check that $G$ acts sharply 2-transitively on its set of involutions.
Since $G$ belongs to $\calc$, it suffices to check 
that no pair of involutions of $G$ is of type $D_\bbZ$ (see Remark \ref{rq_C_et_s2t}).
So consider $(u,v)$ a pair of involutions of $G$, and let $n\in \bbN$ be such that $u,v\in G_n$.
If $(u,v)$ is of type $D_\bbQ$ in $G_n$, it is obviously of type $D_\bbQ$ in $G$.
Otherwise, let $(\tilde u,\tilde v)$ be a maximal pair of type $D_\bbZ$ in $G_n$ such that $\grp{u,v}\subset \grp{\tilde u,\tilde v}$.
Since all embeddings preserve maximality, at each step, the pair $(\tilde u,\tilde v)$ remains maximal unless it becomes of type $D_\bbQ$.
Then Step 1 of the construction ensures that $(\tilde u,\tilde v)$ and $(u,v)$ become of type $D_\bbQ$ at some step. This shows that $(u,v)$
is of type $D_\bbQ$ in $G_{n+1}$, hence also in $G$.
This shows that $G$ acts sharply 2-transitively on its set of involutions. It also follows that all translations of $G$ are conjugate.

Since $G$ belongs to $\calc'$, there remains to show that $G$ has no isolated element and that all homotheties of $G$ are conjugate.

Assume by contradiction that there exists $g\in G$ a maximal isolated element, (equivalently,
$\grp{g}$ is malnormal in $G$). Then $g$ is a maximal isolated element in all $G^m_n$, 
but step 2 ensures that $g$ is conjugate in $\grp{u_0,v_0}$ in $G_{n+1}$, contradicting that $g$ is isolated in $G$.

Similarly, step 3 ensures that all homotheties of $G_n$ become conjugate in $G_{n+1}$.
Since any homothety of $G$ is a homothety in some $G_n$, this shows that all homotheties of $G$ are conjugate.\end{proof}

Having only four conjugacy classes, the group $G$ above is not far from being simple, but may still fail to be so. 
%One can easily ensure simplicity of $G$ as follows (see also \cite{AT21}).
The following result shows that we can additionally ensure that $G$ is simple (see also \cite{AT21}).

\begin{te}\label{thm_simple}
Given a countable group $G_0$ belonging to $\calc'$,
there exists a countable group $G$ containing $G_0$ such that 
\begin{itemize}
    \item $G$ is sharply 2-transitive and belongs to $\calc'$,
    \item $G$ has exactly four conjugacy classes: 
the trivial element,  the set of involutions, the set of translations and the set of homotheties,
\item and $G$ is simple.
\end{itemize}
\end{te}

Taking for instance $G_0=\bbQ^*\ltimes \bbQ$, we get the following result.

\begin{co}\label{co_simple}
There exists a countable sharply 2-transitive group in the class $\calc'$
which is simple and has exactly four conjugacy classes.\qed
\end{co}

\begin{proof}[Proof of Theorem \ref{thm_simple}]
We claim that one can construct a group $G'_0$ containing $G_0$ and in the class $\calc'$ such that
there are four distinct involutions $u_1,u_2,u_3,u_4\in G'_0$ such that $h=u_1u_2u_3$ and $h'=u_1u_2u_3u_4$ are two homotheties.

If the claim holds, then one can take for $G$ the group %with 4 conjugacy classes 
constructed by applying Theorem \ref{4_classes} to $G'_0$.
Indeed, $h$ and $h'$ are still homotheties in $G$, and let us check that $G$ is simple.
Consider $N$ any non-trivial normal subgroup. If $N$ contains an involution or a translation, then it contains all of them so contains $h$ or $h'$ respectively. 
This shows that in all cases, $N$ contains all homotheties, so $N$ also contains the involution $h'h^{-1}$, hence all involutions, all translations, and $N=G$.

We now prove the claim. The group $G_0*\grp{t}$ belongs to $\calc'$ by Example \ref{ex_produit_libre}, and if $u,v,w,z\in G_0$ are 4 distinct involutions then the element $h=uv(twt\m)$ is isolated, and Proposition \ref{cor_HNN_isolated} yields a larger group $G_1$ in which $h$ is a homothety. Similarly, the free product $G_1*\grp{s}$ belongs to $\calc'$, the element $h'=h(szs\m)\in G_1*\grp{s}$ is isolated, and we can embed $G_1*\grp{s}$ using Proposition \ref{cor_HNN_isolated} into a larger group $G'_0$ to ensure that $h'$ is a homothety. This proves our initial claim.
\end{proof}

\section{Getting finite generation}\label{sec_fg}

\subsection{Small cancellation over relatively hyperbolic groups}\label{sec_SC}

Let $G$ be a group hyperbolic relative to a collection of subgroups $\mathcal{H}=\lbrace H_{\lambda}\rbrace_{\lambda\in \Lambda}$.
Recall that an element $g\in H$ is \emph{parabolic} if it is conjugate in some $H_\lambda$,
and \emph{hyperbolic} if it is not parabolic and of infinite order (we note that in \cite{Osi10}, hyperbolic elements are synonymous with non-parabolic, and may have finite order).
Given a hyperbolic element $h\in G$, its commensurator $\Comm_G(\grp{h})=\{g\in G \ | \ gh^n g\m=h^{\pm n} \text{ for some } n\in\mathbb{N}^{\ast}\}$ %is the commensurator of $\grp{h}$. It 
is the unique maximal virtually cyclic subgroup of $G$ containing $h$. Following \cite{Osi10}, we use the following technical definition.

\begin{de}
Let $G$ be a relatively hyperbolic group. We say that a subgroup $H\subset G$ is \emph{suitable} if it contains two hyperbolic elements $h_1,h_2\in H$ (of infinite order) such that $\Comm_G(\grp{h_1})\cap \Comm_G(\{h_2\})=\lbrace 1\rbrace$.
\end{de}

We need the following slight refinement of Theorem 2.4 in \cite{Osi10}.

\begin{te}\label{variante}
Let $G$ be a group hyperbolic relative to a subgroup $P$, and $H\subset G$ be a suitable subgroup. Let $t_1,\dots,t_n$ be arbitrary elements of $G$. Then there exists an epimorphism $\eta : G\twoheadrightarrow Q$ such that:
\begin{enumerate}
    \item $Q$ is hyperbolic relative to $\eta(P)$;
    \item $\eta$ is injective in restriction to $P$;
    \item $\eta(H)$ is a suitable subgroup of $Q$;
    \item $\eta(t_1),\dots,\eta(t_n)$ belong to $\eta(H)$;
    \item \label{it_lift} for every finite subgroup $F$ of $Q$, there exists a subgroup $F'\subset G$ isomorphic to $F$ such that $\eta(F')=F$. 
\end{enumerate}
\end{te}

\begin{proof} This is Theorem 2.4 in \cite{Osi10} except for Assertion \ref{it_lift} which is only stated for finite cyclic groups but the argument actually works for all finite groups. One can also refer to \cite[Proposition 6.12]{Cou13} where this assertion is proved for small cancellation quotients of hyperbolic groups, or to \cite[Lemma 4.3]{DG_recognizing} in the context of Dehn fillings.\end{proof}

\subsection{Construction}\label{sec_construction}

The main result of this section is the following Theorem from which we will deduce
 Theorems \ref{thm_fg} and \ref{thm_4classes}.

% In this subsection, we prove Theorems \ref{thm_fg} and \ref{thm_4classes}. These two results follow immediately from Corollary \ref{corollaire_final} below, which will be proved by applying the following theorem to any countable simple sharply 2-transitive group with four conjugacy classes that belongs to the class $\mathcal{C}'$, as constructed in Section \ref{sec_4conj}.

\begin{te}\label{theoreme2}Let $G_0$ be a sharply 2-transitive group belonging to the class $\mathcal{C}'$. Then there exists a 2-generated sharply 2-transitive group $G$ with Kazhdan property (T), such that $G$ contains $G_0$ and every element of $G$ is conjugate to an element of $G_0$.
%Moreover, one can construct such a group $G$ which has property (T).
\end{te}

Theorems \ref{thm_fg} and \ref{thm_4classes} are immediate consequences of the following corollary.

\begin{co}\label{corollaire_final}
There exists an 
%\Vcomment{ajout "infinite"}\Scomment{ok. C'est automatique à cause de la carac=0, mais on peut le laisser si tu veux insister sur ce point}\Vcomment{T'as raison, mais laissons-le pour insister effectivement.}
infinite sharply 2-transitive group $G$ of characteristic 0 with the following properties:
\begin{itemize}
    \item $G$ is generated by two elements,
    \item $G$ has exactly four conjugacy classes,
    \item $G$ is simple,
    \item $G$ has property (T).
\end{itemize}
\end{co}

\begin{proof}[Proof of the corollary]
Start with a countable simple sharply 2-transitive group $G_0$ of characteristic 0 with four conjugacy classes and 
belonging to $\calc'$, whose existence is proved in Theorem \ref{thm_simple}. 
Let $G$ be the 2-generated group provided by Theorem \ref{theoreme2}. Then $G$ has at most four conjugacy classes,
and at least four by Proposition \ref{optimal}. 
Simplicity of $G$ follows immediately from the simplicity of $G_0$ since all conjugacy classes of $G$ intersect $G_0$.
%To prove that $G$ is simple, consider a non-trivial normal subgroup $N\subset G$, and $g\in N\setminus\{1\}$. Since $g$ has a conjugate in $G_0$, then $N\cap G_0$ is a non-trivial normal subgroup of $G_0$.
%Since $G_0$ is simple, $G_0\subset N$, and since all conjugacy classes of $G$ intersect $G_0$, $N=G$.
\end{proof}

The following notation will be convenient.
\begin{de}
Consider a group $G$ with two subgroups $P,H\subset G$. 
We say that $(G,P,H)$ satisfies $(*)$ if the following hold:
\begin{enumerate}
    \item $G$ is a group in the class $\calc'$,
    \item $G$ is hyperbolic relative to $P$,
    %\item $P$ contains a pair of involutions of type $D_{\bbQ}$ and a subgroup isomorphic to $\bbZ\times (\mathbb{Z}/2\mathbb{Z})$,
    \item $H$ is a suitable subgroup of $G$.
\end{enumerate}
\end{de}

The proof of Theorem \ref{theoreme2} is based on the following result which will be applied iteratively.

\begin{prop}\label{prop_star}
Consider a group $G$ with two subgroups $P,H\subset G$ such that $(G,P,H)$ satisfies $(*)$. Let $E\subset G$ be a finite or virtually cyclic subgroup of $G$. Then there exists an epimorphism $\eta:G\onto \bar G$ which is injective on $P$,
such that $(\bar G,\eta(P),\eta(H))$ satisfies $(*)$ and such that $\eta(E)\subset \eta(H)$ 
and $\eta(E)$ is conjugate to a subgroup of $\eta(P)$. Moreover, every involution of $\bar G$ is the image of an involution of $G$.\end{prop}

We first deduce Theorem \ref{theoreme2} from Proposition \ref{prop_star}. 

% We will use the following standard fact.

% \begin{lemme}\label{prop_T}There exists a non-elementary torsion-free 2-generated hyperbolic group that has property (T).\end{lemme}

% \begin{proof} \Vcomment{J'ai un peu raccourci, et changé de ref}
% Start with any torsion-free hyperbolic group $\Gamma$ with property (T) (for instance, 
% a cocompact torsion-free lattice $\mathrm{Sp}(n,1)$ for $n\geq 2$\Vcomment{il semble qu'il faille $n\geq 2$}). 
% Since any pair of torsion-free hyperbolic groups have a common quotient which is torsion-free hyperbolic 
% \cite{Olshanskii, Champetier}, 
% such a common quotient $\Gamma'$ of $\Gamma$ and the free group $F_2$ satisfies our needs.
% % The group $\Gamma$ has property (T) as $\mathrm{Sp}(n,1)$ has property (T), and $\Gamma$ is hyperbolic. 
% %If $m=2$ there is nothing to do. If $m>2$, we can assume without loss of generality that $\langle x_1,x_2\rangle$ is a suitable subgroup of $\Gamma$, and Theorem \ref{variante} gives a non-elementary torsion-free hyperbolic quotient $Q$ of $\Gamma$ that is generated by the images of $x_1$ and $x_2$, and $Q$ has property (T) as a quotient of $\Gamma$.
% \end{proof}

\begin{proof}[Proof of Theorem \ref{theoreme2}]
    Let $G_0$ be a sharply 2-transitive group belonging to $\calc'$. 
    %\Vcomment{il me semble que ca manquait}
    %To ensure that $G_0$ contains a subgroup isomorphic to
    %$(\bbZ/2\bbZ)\times \bbZ$, we may define $G'_0=G_0*_{\grp{u}} (\grp{u}\times \bbZ)$ where $u\in G_0$ is an involution;
%    by Proposition \ref{prop_gen}, this still belongs to $\calc'$.
 %   \Vcomment{modif}
  Let $H_1$ be a 2-generated torsion-free hyperbolic group with Kazhdan property (T);
    the existence of such group is standard: for instance, take a common torsion-free hyperbolic quotient 
    (\cite{Olshanskii},\cite[Th. 5.19]{Champetier}) of a free group $F_2$ and a torsion-free cocompact lattice in $\mathrm{Sp}(2,1)$.
    Consider the free product $G_1=G_0*H_1$.
    %be its free product with any non-elementary torsion-free 2-generated hyperbolic group $H_1=\langle x,y\rangle$. We can take such a group $H_1$ that has property (T) (see Lemma \ref{prop_T} above).
    It is hyperbolic relative to $P_1:=G_0$, $H_1$ is a suitable subgroup of $G_1$, and it is in the class $\calc'$ by Proposition \ref{prop_gen}, so $(G_1,P_1,H_1)$ satisfies $(*)$.
    
    Let $E_1,\dots,E_n,\dots$ be an enumeration of all the virtually cyclic subgroups of $G_1$ (including all cyclic groups of order 2 or $\infty$).
    We construct inductively a chain of quotients using Proposition \ref{prop_star} as follows.
    Apply Proposition \ref{prop_star} to $(G_1,P_1,H_1)$ and $E_1\subset G_1$, and denote by $G_2=\bar G_1$
    the obtained quotient and by $\eta_1:G_1\onto G_2$ the quotient map. Define $H_2=\eta(H_1)$
    and $P_2=\eta(P_1)\simeq G_0$ so that the image of $E_1$ is contained in $H_2$ and conjugate in $P_2$.
    Since $(G_2,P_2,H_2)$ satisfies $(*)$, one can repeat the argument and obtain a chain of quotients
    $$G_1\xonto{\eta_1} G_2\xonto{\eta_2}\dots \xonto{\eta_{n-1}} G_n \xonto{\eta_n} \dots$$
    and subgroups $H_n=\eta_{n-1}(H_{n-1})$, $P_n=\eta_{n-1}(P_{n-1})\simeq G_0$ of $G_n$
    such that $(G_n,P_n,H_n)$ satisfies $(*)$ and such that the image of $E_1,\dots, E_{n-1}$ in $G_n$
    are contained in $H_n$ and conjugate in $P_n$.
    
    Let $G_\infty$ be the direct limit of this chain and denote the corresponding epimorphism by $\eta_\infty:G_1\ra G_\infty$. Denote $H_\infty=\eta_\infty(H_1)$ and $P_\infty=\eta_\infty(P_1)\simeq G_0$.
    Any element $g\in G_1$ is contained in some $E_j$, so its image in $G_{j+1}$ lies in $H_{j+1}$, hence $\eta_\infty(g)\in H_\infty$.
    This shows that $\eta_{\infty|H_1}$ is onto, so $G_\infty$, as a quotient of $H_1$, has property (T) and is 2-generated.
    Similarly, if $g\in E_j$, then its image in $G_{i+1}$ is conjugate in $P_{j+1}$, which shows that every element of $G_\infty$ has a conjugate in $P_\infty$.
    
    We finally check that $G_\infty$ acts sharply 2-transitively on its set of involutions. 
    Let $u,v\in G_\infty$ be a pair of involutions, and let us prove that it is conjugate to a pair of involutions of $P_\infty$. 
    There exist $n$ and preimages $u_n,v_n\in G_n$ of $u,v$ such that
    $u_n,v_n$ are involutions. Since every involution of $G_i$ is the image of an involution of $G_{i-1}$ 
    under $\eta_{i-1}$, there exist involutions $u_1,v_1\in G_1$ that respectively map to $u,v\in G_\infty$.
    Since the elementary group $\grp{u_1,v_1}=E_j$ for some index $j$, its image in $G_{j+1}$ 
    is conjugate in $P_{j+1}$ hence $\grp{u,v}$ is conjugate in $P_\infty$.
    Since $P_\infty\simeq G_0$ acts transitively on its pairs of involutions, so does $G_\infty$.
    
    It remains to check that the centralizer of a pair of distinct involutions $u,v\in G_\infty$ is trivial.
    If $z\in G_\infty$ centralizes $u\neq v$, then there exist $n$ and lifts $u_n,v_n,z_n\in G_n$ of $u,v,z$ such that
    $u_n\neq v_n$ are involutions and $z_n$ centralizes $u_n,v_n$. Since $G_n$ belongs to $\calc$, it acts freely on its pairs of involutions so $z_n=1$ and $z=1$ which concludes the proof.
%    Note that since $H_1=\langle x,y\rangle$ has property (T), then $G_{\infty}$ has property (T), since property (T) is preserved by taking quotients.
\end{proof}

There remains to prove Proposition \ref{prop_star}.

\begin{proof}[Proof of Proposition \ref{prop_star}]
Consider $(G,P,H)$ satisfying $(*)$. 
In a first step, we are going to embed $G$ in a group $G_1$ 
so that $E$ is conjugate to a subgroup of $P$ in $G_1$.
If $E$ is already parabolic in $G$, we let $G_1=G$ so assume otherwise.
This implies that $E$ is infinite because, since $G$ belongs to $\calc'$, every finite subgroup of $G$ has order at most 2 (Lemma \ref{lem_fini2})
and all involutions are conjugate (Definition \ref{dfn_class_C}).

Since $G$ is relatively hyperbolic, 
the commensurator $\hat E$ of $E$ is virtually cyclic infinite. By Lemma \ref{lem_fini2}, 
$\hat E$ is isomorphic to $\bbZ$, $D_\bbZ$ or $\bbZ\times (\bbZ/2\bbZ)$.
Note that if $\grp{h}$ is a maximal infinite cyclic subgroup of $\hat E$, then $h$ is isolated, a translation of type $D_\bbZ$, or a homothety accordingly.

Consider $E'$ a subgroup of $P$ isomorphic to $\hat E$: 
if $E$ is isomorphic to $\bbZ$ or $\bbZ\times(\bbZ/2\bbZ)$, Lemma \ref{lemme_exists_homothetie} ensures
that $E'$ exists, and if $\hat E\simeq =\grp{u,v}\simeq D_\bbZ$, 
we choose $E'=\grp{u_0,v_0}$ where $(u_0,v_0)$ is a pair of involutions of type $D_\bbQ$, 
which exists by Definition \ref{dfn_class_C}. 
Let $\sigma:\hat E\ra E'$ be an isomorphism and let $G_1$ be the HNN extension $G_1=\grp{G,t \ | \ txt\m=\sigma(x), \ x\in \hat E}$, and let us check that $(G_1,P,H)$ satisfies $(*)$.

The group $G_1$ belongs to $\calc'$: 
\begin{itemize}
    \item if $\hat E=\grp{h}\simeq \bbZ$, this follows from Corollary \ref{cor_HNN_isolated};
    \item if $\hat E \simeq D_\bbZ$, this follows from Corollary \ref{cor_HNN_paire} (where  $u,v\in \hat E$ are two involutions that generate $\hat E$ and $E'=\grp{u_0,v_0}\subset P$ is chosen so that $(u_0,v_0)$ has type $D_\bbQ$);
    \item if $\hat E= \grp{h}\times \grp{u}\simeq \bbZ\times (\bbZ/2\bbZ)$, then this follows from Corollary \ref{cor_HNN_homothetie}.
\end{itemize} 

The group $G_1$ is hyperbolic relative to $P$ by
Dahmani's combination theorem \cite{Dah_combination} as stated in \cite[Theorem 2.5]{Osi10}
where we view $G$ as hyperbolic relative to $\{P,\hat E\}$ (see for instance \cite[Th 2.1]{Osi10}).

%\Vcomment{il y a avait une phrase disant que les involutions sont conjugues dans $P$. Mais a ce stade, on sait que $G_1$ est dans $\calc'$, dont toutes les involutions sont conjuguees donc c'est automatique, n'est-ce pas ?}
%Any involution of $G_1$ is conjugate in the base group $G$ of the HNN extension, and therefore conjugate in $P$.
%Assertion 3 of condition (*) is obvious for $(G_1,H,P)$.

We now check that $H$ is suitable in $G_1$.
By \cite[Lemma 2.3]{Osi10}, $H$ contains 
infinitely many non-commensurable hyperbolic elements $h_1,h_2,\ldots\in H$ 
of infinite order such that $\Comm_G(\grp{h_i})=\grp{h_i}$.
Up to discarding at most two elements, we may assume
that no power of any $h_i$ is $G$-conjugate in $\hat E$ or $\hat E'$.
To check that $\Comm_{G_1}(\grp{h_i})=\grp{h_i}$,
consider $G_1\actson T$ be the Bass-Serre tree of the HNN extension defining $G_1$,
and $x\in T$ a vertex with stabilizer $G$.
For all $k\geq 1$, the set of fixed points of $h_i^k$ in $T$ is exactly $\{x\}$.
It follows that $\Comm_{G_1}(\grp{h_i})$ fixes $x$ so $\Comm_{G_1}(\grp{h_i})=\Comm_{G}(\grp{h_i})=\grp{h_i}$.
This shows that $H$ is suitable in $G_1$.

We thus have embedded $G$ in a group $G_1$ such that $(G_1,P,H)$ satisfies $(*)$, and such that $E$ is conjugate to a subgroup of $P$ in $G_1$.

%We now use Theorem \ref{variante} to 
%construct a quotient $\eta:G_1\onto \bar G$ of $G_1$ such that $\eta_{|G}$ is still surjective.
Let $\{t_1,t_2\}$ be a generating set of $E$, and recall that $t$ denotes the stable letter of the
HNN extension defining $G_1$; in the case where we defined $G=G_1$, we let $t=1$.
Let $\eta_1:G_1\onto \bar G$ be the quotient of $G_1$ given by Theorem \ref{variante} applied to the elements $t,t_1,t_2\in G_1$.
We denote by $\eta:G\ra\bar G$ the restriction of $\eta_1$.
Since $G_1$ is generated by $G$ and $t$, and since $\eta_1(t)\in \eta_1(H)\subset \eta_1(G)$,
it follows that $\eta:G\ra \bar G$ is onto.
%Since $\eta$ coincides with $\eta_1$ on $G_1$,
Theorem \ref{variante} says that $\eta$ is injective in restriction to $P$, 
that the group $\bar G$ is hyperbolic relative to $\eta(P)$,
that $\eta(H)$ is suitable in $\bar G$,
and that $\eta(E)=\grp{\eta(t_1),\eta(t_2)}$ is contained in $\eta(H)$.
Since $E$ is conjugate in $P$ in $G_1$, $\eta(E)$ is conjugate in $\eta(P)$ in $\bar G$.

It also says that any involution $v\in \bar G$ is the image of some involution $\tilde v$ in $G_1$.
It follows that $v$ is the image of some involution in $G$: $\tilde v$ is conjugate to some involution $\tilde v'\in G$,
and since $\eta$ is onto, there exists $g\in G$ such that $\eta(g\tilde v'g\m)=v$.
We note that since all involutions of $G$ are conjugate to each other, 
this is also the case in  $\bar G$.
%We conclude the proof by checking that every involution $\bar u\in \bar G$ is the image of an involution of $G$.
%By Theorem \ref{variante}(5), there exists an involution $u_1\in G_1$ that maps to $\bar u$ under $\eta_1$.
%We are done if $G=G_1$. Otherwise, $G_1$ being an HNN extension, $u_1$ is conjugate to an involution $u'\in G$.
%Thus $\eta(u')=\bar u^{\bar g}$ for some $\bar g\in \bar G$ and since $\eta$ is onto,
%$\bar u$ is the image of a conjugate of $u'$, namely $\bar u=\eta(u'^{\tilde g\m})$ where $\tilde g\in G$ is such that $\eta(\tilde g)=\bar g$.

There remains to check that $(\bar G, \eta( P),\eta(H))$ satisfies $(*)$, and the only remaining fact to prove 
is that $\bar G$ belongs to $\calc'$.
%Abusing notations, we identify $P$ with its image and denote $P$ instead of 
In what follows we denote $\bar P=\eta(P)$.

%To check that $(\bar G, \bar P,\eta(H))$ satisfies $(*)$, it remains to check that $\bar G$ belongs to $\calc'$. 
We first characterize finite and virtually cyclic subgroups of $\bar G$.
Since $G_1$ belongs to $\calc'$, every non-trivial finite subgroup of $G_1$ is isomorphic to $\bbZ/2\bbZ$ (see Lemma \ref{lem_fini2}),
and the same holds for $\bar G$ by Assertion (5) of Theorem \ref{variante}.
By Lemma \ref{lem_fini2}, it follows that every infinite virtually cyclic subgroup of $\bar G$ is isomorphic to $\bbZ$, $D_\bbZ$ or $\bbZ\times (\bbZ/2\bbZ)$.

We now check that $\bar G$ belongs to $\calc$ (Definition \ref{dfn_class_C}).
The following consequence of relative hyperbolicity will be useful: $\bar P$ is almost malnormal in $\bar G$,
\ie if $\bar P^g\cap \bar P$ is infinite, then $g\in \bar P$ (see for instance Lemma 8.3 in \cite{Osi10}).

First, for any pair of involutions $(u,v)\in\cali_{\bar G}^{(2)}$, 
the group $\grp{u,v}$ is infinite dihedral since finite subgroups of $\bar G$ have order at most 2.
We already noticed that $\bar G$ acts transitively on its involutions.
Let $(u,v)\in\cali_{\bar G}^{(2)}$ be a pair of involutions, and let us check that its centralizer $Z=Z_{\bar G}(\grp{u,v})$ is trivial and that it is
either or type $D_\bbZ$ or $D_\bbQ$.

%To check that $\bar G$ acts freely on $\cali_G^{(2)}$, it suffices to check that for any pair of  distinct involutions $u,v$, 
%the centralizer $Z$ of $\grp{u,v}$ is trivial. 

If $\grp{u,v}$ is not parabolic, then its commensurator 
$\Comm_{\bar G}(\grp{u,v})\supset {\bar G}(\grp{u,v})$ is virtually cyclic, necessarily isomorphic to $D_\bbZ$, 
so $(u,v)$ is of type $D_\bbZ$ and its centralizer $Z$ is trivial.
%so $Z=\{1\}$. This also shows that $\grp{u,v}$ is of type $D_\bbZ$.

If $\grp{u,v}$ is parabolic in $\bar G$, we may assume that $\grp{u,v}\subset \bar P$. 
Then $Z$ is contained in $\bar P$ by almost-malnormality of $\bar P$, so $Z$ is trivial because $\bar P$ is isomorphic to the subgroup $P$ of $G$, and $G$ belongs to $\calc$.
Denote by $\tilde u,\tilde v\in P$ the preimages of $u,v$ by the isomorphism $\eta_{|P}:P\xra{\sim} \bar P$.

%If $(u,v)$ is of type $D_\bbQ$ in $\bar P$, it also is in $\bar G$.
%Applying Lemma \ref{lem_C_relhyp} to $(G,P)$, we have that 
If $(\tilde u,\tilde v)$ is of type $D_\bbQ$ in $G$,
then $\grp{\tilde u,\tilde v}$ is contained in a subgroup $\tilde D$ of $G$ isomorphic to $D_\bbQ$.
Such a group $\tilde D$ has to be parabolic, hence conjugate in $P$.
It follows that $(u,v)$ is of type $D_\bbQ$ since $\grp{u,v}\subset \eta(\tilde D)\simeq D_\bbQ$.
If $(\tilde u,\tilde v)$ is of type $D_\bbZ$ in $G$, then
$\Comm_{P}(\grp{\tilde u,\tilde v})\simeq D_\bbZ$. Using the isomorphism $\eta$,
$\Comm_{\bar P}(\grp{u,v})\simeq D_\bbZ$.
By almost malnormality of $\bar P$, we get that $\Comm_{\bar G}(\grp{u,v})=\Comm_{\bar P}(\grp{u,v})$
so $(u,v)$ is of type $D_\bbZ$ in $\bar G$.

To prove that $\bar G$ belongs to $\calc$, there remains to prove that $\bar G$ acts transitively
on the set of pairs $(u,v)\in\cali_{\bar G}^{(2)}$ of type $D_\bbQ$.
But any such pair is conjugate in $\bar P$ because a group isomorphic to $D_\bbQ$ has to be parabolic,
so there remains to check that $\bar P$ acts transitively on its set of pairs $(u,v)\in\cali_{\bar P}^{(2)}$ of type $D_\bbQ$.
Viewing $\bar P\simeq P$ as a subgroup of $G$, % which is assumed to belong to $\calc$,
we know that any two pairs in $\cali_{P}^{(2)}$ of type $D_\bbQ$ are in the same $G$-orbit.
Since $P$ is almost malnormal in $G$, they are actually in the same $P$-orbit.
This concludes the proof that $\bar G$ is in the class $\calc$.

To check that $\bar G$ belongs to $\calc'$, there remains
to show that any element $ h\in\bar G$ of infinite order is a translation, a homothety or is isolated.
If $\grp{h}$ is not parabolic, this follows from the classification of elementary subgroups
(the three cases occurring when $\Comm_{\bar G}(\grp{h})$ is isomorphic to $D_\bbZ$, $\bbZ\times (\bbZ/2\bbZ)$
or $\bbZ$ respectively).
Thus, we may assume that $h\in \bar P$. 
Let $\tilde h\in P$ be its preimage in $P$ under the isomorphism $\eta:P\xra{\sim} \bar P$.
If $\tilde h$ is a translation or a homothety in $G$, it stays so in $\bar G$ because $\eta$ does not kill any involution.
If $\tilde h$ is isolated in $G$, it is isolated in $P$ and $h$ is isolated in $\bar P$.
Its commensurator $\Comm_{\bar G}(\grp{h})$ is contained in $\bar P$ by almost malnormality. It follows that $h$ is isolated in $\bar G$, which proves that $\bar G$ belongs to $\calc'$ and concludes the proof.
\end{proof}

\renewcommand{\refname}{Bibliography}
\bibliographystyle{alpha}
\bibliography{biblio2}

\newpage

\setlength{\parindent}{0pt}
Simon André
\\
Institut für Mathematische Logik und Grundlagenforschung, 
\\
WWU Münster, Einsteinstraße 62 48149 Münster, Germany.
\\
E-mail address: \href{mailto:sandre@wwu.de}{sandre@uni-muenster.de}

\medskip

Vincent Guirardel
\\
Univ Rennes, CNRS, IRMAR - UMR 6625, F-35000 Rennes, France.
\\
E-mail address: \href{mailto:vincent.guirardel@univ-rennes1.fr}{vincent.guirardel@univ-rennes1.fr}

\end{document}